\documentclass{article}

\usepackage[margin=1in]{geometry}
\usepackage[all]{xy}
\usepackage{amssymb}
\usepackage{amsfonts}
\usepackage{amsthm}
\usepackage{amsmath}
\usepackage{hyperref}
\usepackage{mathtools}
\usepackage{ytableau}
\usepackage{url}
\mathtoolsset{showonlyrefs}
\usepackage{tikz}

\newtheorem*{thm*}{Theorem}
\newtheorem{thm}{Theorem}[section]

\newtheorem{lem}[thm]{Lemma}  

\newtheorem{cor}[thm]{Corollary}

\theoremstyle{definition}

\theoremstyle{remark}
\newtheorem{remark}{Remark}

\newtheorem*{question*}{Question}

\newcommand{\ord}{\mathrm{ord}(\chi)}
\newcommand{\YY}{\mathbb{Y}}

\newcommand{\ZZ}{\mathbb{Z}}
\newcommand{\FF}{\mathbb{F}}
\newcommand{\GAP}{\mathrm{GAP}}
\newcommand{\AP}{\mathrm{AP}}
\newcommand{\FactType}{\omega}
\newcommand{\Prob}{\mathrm{P}}
\newcommand{\Ind}{\mathbf{1}}

\numberwithin{equation}{section}

\title{Mean values of arithmetic functions in short intervals and in arithmetic progressions in the large-degree limit}
\date{}
\author{Ofir Gorodetsky}

\newcommand{\Addresses}{{
		\bigskip
		\footnotesize
		
		\textsc{Raymond and Beverly Sackler School of Mathematical Sciences, Tel Aviv University, Tel Aviv 69978, Israel}\par\nopagebreak
		\textit{E-mail address:} \texttt{ofir.goro@gmail.com}
}}

\begin{document}
\maketitle

\begin{abstract}
A classical problem in number theory is showing that the mean value of an arithmetic function is asymptotic to its mean value over a short interval or over an arithmetic progression, with the interval as short as possible or the modulus as large as possible. 

We study this problem in the function field setting, and prove for a wide class of arithmetic functions (namely factorization functions), that such an asymptotic result holds, allowing the size of the short interval to be as small as a square-root of the size of the full interval, and analogously for arithmetic progressions. 

For instance, our results apply for the indicator function of polynomials with a divisor of given degree, and are much stronger than those known for the analogous function over the integers.

As opposed to many previous works, our results apply in the \emph{large-degree limit}, where the base field $\FF_q$ is fixed. Our proofs are based on relationships between certain character sums and symmetric functions, and in particular we use results from symmetric function theory due to E\u{g}ecio\u{g}lu and Remmel. We also use recent bounds of Bhowmick, L\^{e} and Liu on character sums, which are in the spirit of the {D}rinfeld--{V}l\u{a}du\c{t} bound.
\end{abstract}

\section{Introduction}
An important and challenging class of problems in analytic number theory is showing that the mean value of an arithmetic function in a short interval $[x,x+h]\cap \ZZ$ is asymptotic to its mean over a long interval. Concretely, this often takes the following form: proving for various $\alpha\colon \mathbb{N}\to \mathbb{C}$ that
\begin{equation}\label{eq:meanLongShort}
\frac{\sum_{x \le n \le x+x^{\varepsilon}} \alpha(n)}{x^{\varepsilon}} = \frac{\sum_{n \le x} \alpha(n)}{x}(1+o(1)), \qquad x \to \infty
\end{equation}
for $\varepsilon\in (0,1)$ as small as possible. If the mean value of $\alpha$ tends to $0$ and does not have a main term, as happens, for example, for the M\"obius function, one should be more careful and instead ask for 
\begin{equation}\label{eq:meanLongShort2}
\frac{\sum_{x \le n \le x+x^{\varepsilon}} \alpha(n)}{x^{\varepsilon}} = o(1), \qquad x \to \infty
\end{equation}
to hold for $\varepsilon \in (0,1)$ as small as possible. A related class of problems is showing that the mean value of $\alpha\colon \mathbb{N}\to \mathbb{C}$ over integers $1\le n\le x$ which lie in the arithmetic progression $a \bmod q$ ($\gcd(a,q)=1$) is close to the mean value of $\alpha$ over integers $1\le n \le x$ which are coprime to $q$. Formally, this often means proving
\begin{equation}\label{eq:meanAP}
\frac{\sum_{n \le x, \, n \equiv a \bmod q}\alpha(n)}{x/\phi(q)} = \frac{\sum_{ n \le x, \, (n,q)=1} \alpha(n)}{x(\phi(q)/q)} (1+o(1)), \qquad x \to \infty
\end{equation}
or
\begin{equation}\label{eq:meanAP2}
\frac{\sum_{n \le x, \, n \equiv a \bmod q}\alpha(n)}{x/\phi(q)} = o(1), \qquad x \to \infty
\end{equation}
uniformly in $q<x^{1-\varepsilon}$ for $\varepsilon \in (0,1)$ as small as possible.

Conditionally on the Generalized Riemann Hypothesis (GRH), it is a classical result that in the range $\varepsilon>1/2$, \eqref{eq:meanLongShort} and \eqref{eq:meanAP} hold for the von Mangoldt function $\Lambda$ and \eqref{eq:meanLongShort2} and \eqref{eq:meanAP2} hold for the M\"obius function $\mu$. Ramachandra \cite{ramachandra1976} showed that RH implies \eqref{eq:meanLongShort} for $\varepsilon>1/2$ and $\alpha=d_k$, the $k$th divisor function ($k$ not necessarily an integer) and some other functions. K\'{a}tai \cite{katai1985}, building on ideas of Ramachandra, proved that RH also implies \eqref{eq:meanLongShort} for $\varepsilon>1/2$ and $\alpha$ the indicator function of $k$-almost primes. Some other examples exist \cite{cui2014, Feng2017, katai2003}. 

There are close analogues in the function field setting of these results. Because GRH is known in that setting (a theorem due to Weil), they are unconditional. However, both in the number field setting and the function field setting, these results are proven by \emph{ad hoc} arguments, each using GRH in a different way, and in particular the results always require the associated Dirichlet series $\sum \alpha(n)/n^s$ to possess certain structure. In this paper, we prove analogues of \eqref{eq:meanLongShort}--\eqref{eq:meanAP2} in the function field setting for $\varepsilon\in (1/2,1)$ for a wide class of arithmetic functions defined over the polynomial ring $\FF_q[T]$, the so-called ``factorization functions.'' As in previous works, our main tool is GRH for $\FF_q(T)$, but we are able to treat all sensible arithmetic functions at once. This leads to interesting applications which are currently out of reach in the number field setting; these are described in \S \ref{sec:app}.

\subsection{The function field setting}
Throughout we fix a prime power $q$. Let $\FF_q[T]$ be the polynomial ring in $T$ over $\FF_q$, $\mathcal{M}_{n} \subseteq \FF_q[T]$ be the subset of monic polynomials of degree $n$, $\mathcal{P}_{n} \subseteq \mathcal{M}_{n}$ be the subset of monic irreducible polynomials of degree $n$, $\mathcal{M}= \cup_{n \ge 0} \mathcal{M}_{n}$ be the set of all monic polynomials and $\mathcal{P}=\cup_{n \ge 0} \mathcal{P}_{n}$ be the set of all monic irreducible polynomials.

To any $f \in \mathcal{M}$ with prime factorization $f=\prod_{i=1}^{k} P_i^{e_i}$ ($P_i \in \mathcal{P}$ distinct, $e_i \ge 1$), we can associate the following multiset, named the \textit{extended factorization type of $f$}:
\begin{equation}
\FactType_f := \{ (\deg (P_i), e_i) : 1 \le i \le k\}.
\end{equation}
(We often omit the word ``extended.'') Following Rodgers \cite{rodgers2018}, an arithmetic function $\alpha \colon \mathcal{M} \to \mathbb{C}$ is called a \emph{factorization function} if $\alpha(f)$ depends only on $\FactType_f$. Some of the most commonly studied arithmetic functions in number theory, when considered in the function field setting, are instances of factorization functions: the von Mangoldt function
\begin{equation} 
\Lambda(f) = \begin{cases} \deg (P) & \mbox{if $f=P^k$ for $P \in \mathcal{P}$ and $k \ge 1$}, \\ 0 & \mbox{otherwise}, \end{cases}
\end{equation}
the M\"obius function
\begin{equation}
\mu(f) = \begin{cases} (-1)^{k} & \mbox {if $f=P_1P_2\cdots P_k$ for distinct $P_i \in \mathcal{P}$,} \\ 0 & \mbox{otherwise}, \end{cases}
\end{equation}
the indicator of squarefrees
\begin{equation}
\mu^2(f) = \begin{cases} 1 & \mbox {if $f$ is squarefree,} \\ 0 & \mbox{otherwise}, \end{cases}
\end{equation}
the divisor functions, and many more.

A short interval of size $q^{h+1}$ around $f_0 \in \FF_q[T]$ is the subset
\begin{equation}
I(f_0,h):=\{ f \in \FF_q[T] : \|f-f_0\| \le q^{h+1} \} = \{ f \in \FF_q[T] : \deg(f-f_0) \le h\},
\end{equation}
where $\|f\|:=|\FF_q[T]/(f)|=q^{\deg (f)}$ for $f \neq 0$ and $\|0\|:=0$. In addition to short intervals in $\FF_q[T]$, we also recall the definition of arithmetic progressions in this setting. Consider a non-zero polynomial $M\in \FF_q[T]$ and a polynomial $f_0 \in \mathcal{M}_n$ coprime to $M$. We define the arithmetic progression of monic polynomials of degree $n$ in the residue class $f_0 \bmod M$ to be
\begin{equation}
\AP(f_0,M) := \{ f \in \mathcal{M}_{\deg(f_0)} : f \equiv f_0 \bmod M \},
\end{equation}
which is a subset of 
\begin{equation}
\mathcal{M}_{n;M} := \{ f \in \mathcal{M}_{n} : \gcd(f,M)=1\}.
\end{equation}
If $n \ge \deg(M)$, then $\left| \mathcal{M}_{n;M} \right| = q^{n-\deg(M)} \phi(M)$, where $\phi(M)$ is Euler's totient function.
\subsection{Main results}
For a given arithmetic function $\alpha \colon \mathcal{M} \to \mathbb{C}$, let $\langle \alpha \rangle_{S}$ be the mean value of $\alpha$ over a non-empty finite subset $S \subseteq \mathcal{M}$:
\begin{equation}
\langle \alpha \rangle_{S} := \frac{1}{|S|} \sum_{f \in S} \alpha(S).
\end{equation}
We prove the following analogue of \eqref{eq:meanLongShort} and \eqref{eq:meanLongShort2}.
\begin{thm}\label{cor:ForFunc}
Let $\alpha \colon \mathcal{M}\to \mathbb{C}$ be a factorization function. Let $0 \le h \le n-1$ and $f_0 \in \mathcal{M}_{n}$. Then
\begin{equation}\label{eq:upper_function}
\left| \langle \alpha \rangle_{I(f_0,h)} - \langle \alpha \rangle_{\mathcal{M}_{n}} \right| \le \max_{f \in \mathcal{M}_{n}}  \left|\alpha(f) \right|   q^{\frac{n}{2}-h-1} e^{O_q(\frac{n \log \log (n+2)}{\log (n+2)})}.
\end{equation}
\end{thm}
The analogous result for arithmetic progressions is the following.
\begin{thm}\label{cor:ForFuncAP}
Let $\alpha\colon \mathcal{M}\to \mathbb{C}$ be a factorization function. Let $M \in \FF_q[T]$ be a non-zero polynomial and let $f_0 \in \mathcal{M}_n$ be a polynomial coprime to $M$. Then whenever $n \ge \deg(M)$,
\begin{equation}\label{eq:upper_functionAP}
\left| \langle \alpha \rangle_{\AP(f_0,M)} - \langle \alpha \rangle_{\mathcal{M}_{n;M}} \right| \le \max_{f \in \mathcal{M}_{n;M}} \left|\alpha(f) \right|  q^{\deg(M)-\frac{n}{2}} e^{O_q(\frac{n \log \log (n+2)}{\log (n+2)})}.
\end{equation}
\end{thm}
As long as $\max_{f \in \mathcal{M}_n} \left| \alpha(f) \right|$ grows subexponentially in $n$, Theorems \ref{cor:ForFunc} and \ref{cor:ForFuncAP} give a non-trivial result in the range $\limsup_{n \to \infty} \frac{h+1}{n} > 1/2$ and $\limsup_{n \to \infty} (1-\frac{\deg(M)}{n}) > 1/2$, respectively. These limsups are the function field analogue of the quantity $\varepsilon$ discussed before. 

We also prove the following results on the variance of the sums $\sum_{f \in I(f_0,h)} \alpha(f)$ and $\sum_{f \in \AP(f_0,M)} \alpha(f)$ as $f_0$ varies. First, we define the short-interval variance
\begin{equation}
\mathrm{Var}(\alpha;n,h) := \frac{1}{\left| \mathcal{M}_n \right|} \sum_{f_0 \in \mathcal{M}_n} \left| \sum_{f \in I(f_0,h)} \alpha(f) - q^{h+1} \langle \alpha \rangle_{\mathcal{M}_n}\right|^2,
\end{equation}
and the arithmetic progression variance
\begin{equation}
\mathrm{Var}_M(\alpha;n) := \frac{1}{\left| \mathcal{M}_{n;M} \right|} \sum_{f_0 \in \mathcal{M}_{n;M}} \left| \sum_{f \in \AP(f_0,M)} \alpha(f) - q^{n-\deg(M)} \langle \alpha \rangle_{\mathcal{M}_{n;M}}\right|^2.
\end{equation}
\begin{thm}\label{thm:var}
Let $\alpha \colon \mathcal{M}\to \mathbb{C}$ be a factorization function. For any $0 \le h \le n-1$, we have
\begin{equation}\label{eq:upper_functionvar}
\mathrm{Var}(\alpha;n,h)  \le 
\max_{f \in \mathcal{M}_{n}}  \left|\alpha(f) \right|^2  q^{h+1} e^{O_q(\frac{n \log \log (n+2)}{\log (n+2)})}.
\end{equation}
For any non-zero polynomial $M$ with $\deg(M)\le n$, we have
\begin{equation}\label{eq:upper_functionvarap}
\mathrm{Var}_M(\alpha;n)  \le \max_{f \in \mathcal{M}_{n;M}}  \left|\alpha(f) \right|^2   \frac{q^{n}}{\phi(M)} e^{O_q(\frac{n \log \log (n+2)}{\log (n+2)})}.
\end{equation}
\end{thm}
The trivial bound for $\mathrm{Var}(\alpha;n,h)$ is $q^{2(h+1)}\max_{f \in \mathcal{M}_{n}} |\alpha(f)|^2$, and we beat it as long as  $\limsup_{n \to \infty} \frac{h+1}{n}$ is positive, which corresponds to $\varepsilon>0$ in the number field setting. Similarly, \eqref{eq:upper_functionvarap} beats the trivial bound as long as $\limsup_{n \to \infty} (1-\frac{\deg(M)}{n}) > 0$.
\subsection{Applications}\label{sec:app}
Let $H(x,y,2y)$ be number of positive integers up to $x$ which have a divisor in $(y,2y]$. Ford \cite{ford20082, ford2008} showed that
\begin{equation}
H(x,y,2y)\asymp \frac{x}{(\log y)^{\delta} (\log \log y)^{3/2}}, \qquad 3 \le y \le \sqrt{x},
\end{equation}
where $\delta=1-\frac{1+\log \log 2}{\log 2}$. Here $A \asymp B$ means $cB \le A \le CB$ for absolute positive constants $c$ and $C$. Ford has a short interval version of this result \cite[Theorem~2]{ford20082}:
\begin{equation}\label{eq:FordShort}
H(x,y,2y)-H(x-\Delta,y,2y) \asymp \frac{\Delta}{x} H(x,y,2y), \qquad y_0 \le y \le \sqrt{x}, \ \frac{x}{\log^{10} x} \le \Delta \le x,
\end{equation}
where $y_0$ is a sufficiently large real number. Let $H_q(n,d)$ be the function field analogue of $H(x,y,2y)$, counting monic polynomials of degree $n$ which have a divisor of degree $d$. Meisner \cite[Theorem~1.2]{meisner2018erd} has recently shown that
\begin{equation}
H_q(n,d) \asymp \frac{q^n}{d^{\delta} (1+\log d)^{3/2}}, \qquad 1\le d \le \frac{n}{2}.
\end{equation}
Applying Theorem \ref{cor:ForFunc} with 
\begin{equation}
\alpha_{d} \colon \mathcal{M}\to\mathbb{C}, \qquad \alpha_{d}(f) =\begin{cases} 1 & \mbox{if $f$ has a divisor of degree $d$}, \\ 0 & \mbox{otherwise,} \end{cases}
\end{equation}
we obtain the following
\begin{cor}
For any $n \ge 1$, let $f_0 \in \mathcal{M}_n$ and $1 \le d \le n$. As $n \to \infty$, we have
\begin{equation}
\sum_{\substack{f \in I(f_0,h)\\ f \text{ has a divisor of degree }d}} 1 \sim q^{h+1} \frac{H_q(n,d)}{q^n}
\end{equation}
uniformly in $d$, as long as $\limsup_{n \to \infty} \frac{h}{n}>1/2$. 
\end{cor}
Not only is this corollary valid in a much wider range than \eqref{eq:FordShort}, it is an \emph{asymptotic result} and not an estimate. In the integer setting, RH might help in improving the range where \eqref{eq:FordShort} holds, but it is not clear how to use it to establish an asymptotic formula. We raise the following
\begin{question*}
Assume $RH$. Do we have 
\begin{equation}
H(x,y,2y)-H(x-\Delta,y,2y) = \frac{\Delta}{x} H(x,y,2y) (1+o(1)), \qquad y_0 \le y \le \sqrt{x}, \ \frac{x}{\log^{10} x} \le \Delta \le x,
\end{equation}
as $x$ tends to infinity?
\end{question*}
Our second example is Hooley's Delta function. In his study of Waring's problem, diophantine approximation and other problems \cite{hooley1979}, Hooley introduced the function
\begin{equation}
\Delta(n):= \max_{u>0} \sum_{d \mid n, \, d \in (u,eu]} 1,
\end{equation}
and obtained upper and lower bounds on its mean value. Hall and Tenenbaum \cite{hall1982} showed that
\begin{equation}
\frac{1}{x}\sum_{n \le x} \Delta(n) = O(\log^{A_0}x\sqrt{\log \log x})
\end{equation}
for an explicit (optimal) constant $A_0$, and
\begin{equation}
\frac{1}{x} \sum_{n \le x} \Delta(n) = \Omega( \log \log x).
\end{equation}
As far as the author knows, $\Delta$ was not studied in short intervals or in arithmetic progressions. The function field analogue of $\Delta$ is
\begin{equation}
\Delta_q\colon \mathcal{M} \to \mathbb{C}, \qquad \Delta_q(f) := \max_{0 \le i \le \deg(f)} \sum_{d \mid f, \, \deg(d)=i} 1.
\end{equation}
Let $d_2(f)$ be the usual divisor function, which satisfies $\langle d_2 \rangle_{\mathcal{M}_n} = n+1$ \cite[Proposition~2.5]{rosen2002}. It is evident that $\max_{f \in \mathcal{M}_n} \Delta_q(f) \le \max_{f \in \mathcal{M}_n} d_2(f)$, which is known to grow slower than any power of $q^n$ as $n$ tends to infinity. Also, $\langle \Delta_q \rangle_{\mathcal{M}_n}, \, \langle \Delta_q \rangle_{\mathcal{M}_{n;M}} \ge 1$, since $\Delta_q \ge 1$. Applying Theorems \ref{cor:ForFunc} and \ref{cor:ForFuncAP} with $\alpha=\Delta_q$, we obtain the following
\begin{cor}
For any $n \ge 1$, let $f_0 \in \mathcal{M}_n$. As $n \to \infty$, we have
\begin{equation}
\langle \Delta_q \rangle_{I(f_0,h)} \sim \langle \Delta_q \rangle_{\mathcal{M}_{n}}
\end{equation}
as long as $\limsup_{n \to \infty} \frac{h}{n} > 1/2$, and
\begin{equation}
\langle \Delta_q \rangle_{\AP(f_0,M)} \sim \langle \Delta_q \rangle_{\mathcal{M}_{n;M}}
\end{equation}
as long as $\limsup_{n \to \infty} \frac{\deg(M)}{n} < 1/2$.
\end{cor}

\subsection{Previous works}

\subsubsection{Mean values}
In the large-$q$ limit, where the degree $n$ is fixed and the size of the base field $\FF_q$ tends to infinity, see Bank \emph{et al.} \cite{efrat2015} for a general result on mean values, where completely different methods are employed and do not give a result for growing $n$. 

For some specific functions, one can go beyond some of our mean-value results, see the works of Sawin \cite{sawin2018} and Sawin and Shusterman \cite{sawin2019}.
\subsubsection{Variance}
In the large-$q$ limit, see Rodgers \cite{rodgers2018} for a general result on the variance of arithmetic functions. Again, the methods are different than ours and do not allow one to vary $n$ and fix $q$.

For some specific functions, the work of Ramachandra in the integer setting \cite{ramachandra1976} gives non-trivial mean-value results and variance results, which, although conditional on RH, are much weaker in comparison to ours. For instance, in \cite[Equation~(11)]{ramachandra1976}, an upper bound of the form $O(H^2 X^{-o(1)})$ is obtained for the variance of $\mu$ in short intervals (and some other functions), while our upper bound should be compared with $O(H X^{o(1)})$, which is better whenever $X^{\varepsilon}<H<X^{1-\varepsilon}$. Recently, Matom\"{a}ki and Radziwi\l \l  \,\cite{matomaki2016} proved unconditionally that 
\begin{equation}
\frac{1}{X} \int_{X}^{2X} (\sum_{n \in [x,x+H]} \mu(n))^2 dx = o(H^2)
\end{equation}
which goes beyond Ramachandra's result for $\mu$ if $H$ is very small. The methods of \cite{matomaki2016} apply in general to bounded multiplicative functions, and so they cannot handle e.g. divisor functions or the von Mangoldt function. Our upper bound is qualitatively better in the range $X^{\varepsilon}<H<X^{1-\varepsilon}$.
\section{Results on character sums}
Theorems~\ref{cor:ForFunc}--\ref{thm:var} rest on the character sums estimate given in Theorem~\ref{thm:expothm} below. The theorem involves Hayes characters, which are a function field analogue of Dirichlet characters defined in \S\ref{sec:hayes}. Informally, a function $\chi\colon \mathcal{M}\to\mathbb{C}$ is called a Hayes character modulo $R_{\ell,M}$ ($\ell$ a non-negative integer, $M$ a non-zero polynomial) if the following conditions hold: (1) $\chi(fg)=\chi(f)\chi(g)$ for every $f,g \in \mathcal{M}_q$, (2) $\chi(f)$ depends only on the residue of $f$ modulo $M$ and on the first $\ell$ next-to-leading coefficients of $f$ and (3) $\chi(f)=0$ if and only if $\gcd(f,M) \neq 1$. Such a $\chi$ is called \emph{trivial} if it only assumes the values $0$ and $1$. For instance, the notion of a Dirichlet character modulo $M$ coincides with that of a Hayes character modulo $R_{0,M}$
\begin{thm}\label{thm:expothm}
Let $\alpha\colon \mathcal{M}\to \mathbb{C}$ be a factorization function, and let $\chi$ be a non-trivial Hayes character modulo $R_{\ell,M}$. We have 
\begin{equation}\label{eq:upper}
\left| \sum_{f \in \mathcal{M}_n}  \alpha(f)\chi(f) \right| \le 7\max_{f \in \mathcal{M}_{n;M}}\left|\alpha(f)\right| q^{\frac{n}{2}} \binom{20(\ell+\deg(M)+1)+n-1}{n}.
\end{equation}
If $\ell+\deg(M)  =O(n)$, then
\begin{equation}\label{eq:upper2}
\left| \sum_{f \in \mathcal{M}_n} \alpha(f) \chi(f) \right| \le \max_{f \in \mathcal{M}_{n;M}}\left|\alpha(f)\right| q^{\frac{n}{2}} e^{O_q(\frac{n \log \log(n+2)}{\log(n+2)})}.
\end{equation}
\end{thm}
\begin{remark}
If $\ell+\deg(M)$ is small compared to $n$, say $\ell+\deg(M)=o(n/\ln n)$, then \eqref{eq:upper} is superior to \eqref{eq:upper2}. For $\ell+\deg(M)$ which is proportional to $n$, \eqref{eq:upper2} is better than \eqref{eq:upper}.
\end{remark}
We manage to prove such a theorem, which works for any $\alpha$, by reducing it to estimating a single character sum, which we now describe.

Let $\Omega$ be the set of finite multisets of elements from $\mathbb{N} \times \mathbb{N}$, so that $\omega_f$, the factorization type of a polynomial $f$, is an element of $\Omega$. For an element $\FactType=\{ (d_i,e_i) : 1 \le i \le k\} \in \Omega$, we define its size to be $|\FactType|:=\sum_{i=1}^{k} d_i e_i$ and its length to be $\ell(\FactType):=k$. For a factorization function $\alpha$ and a factorization type $\omega \in \Omega$, we denote by $\alpha(\omega)$ the value of $\alpha$ on a polynomial $f$ with $\omega_f=\omega$ if such a polynomial exists, and otherwise set $\alpha(\omega)=0$. We have, by the triangle inequality,
\begin{equation}\label{eq:triang}
\left| \sum_{f \in \mathcal{M}_n} \alpha(f) \chi(f) \right| = \left| \sum_{\substack{\omega \in \Omega \\ |\omega|=n }} \alpha(\omega)\sum_{\substack{f \in \mathcal{M}_n\\ \omega_f=\omega}} \chi(f) \right| \le \max_{f \in \mathcal{M}_{n;M}} \left| \alpha(f) \right| \sum_{\substack{\omega \in \Omega \\ |\omega|=n}}  \left|\sum_{\substack{f \in \mathcal{M}_n \\ \omega_f = \omega }}  \chi(f) \right|.
\end{equation}
Thus, it suffices to bound the sum on the right-hand side of \eqref{eq:triang}. In order to bound the character sums
\begin{equation}
\sum_{f \in \mathcal{M}_{n}:\, \FactType_f=\FactType} \chi(f),
\end{equation}
we relate them to products of \emph{monomial symmetric polynomials} evaluated at $\chi(P)$ where $P$ runs over irreducible polynomials of certain degrees, see Lemmas~\ref{lem:multi} and \ref{lem:monomial}. We use tools from symmetric function theory in order to be able to use RH efficiently in bounding these evaluations of symmetric polynomials.

Apart from symmetric function theory and RH, we use an additional idea, which first appeared in \cite{bhowmick2017}. Namely, we use the trivial bound in place of RH in some cases, as the trivial bound is superior when the degrees considered are small. This is in the spirit of the Drinfeld--Vl\u{a}du\c{t} bound, where one bounds the number of points on curves over $\FF_q$ when $q$ is fixed and the genus varies.	

\section{Hayes characters}\label{sec:hayes}
Here we review the function field analogue of Dirichlet characters, first introduced by Hayes in the paper \cite{hayes1965} which is based on his thesis. We call these characters ``Hayes characters," or sometimes ``generalized arithmetic progression characters." Unless otherwise stated, the proofs of the statements in this section appear in Hayes' original paper. The main difference between Hayes characters and Dirichlet characters is that in the function field setting we can also consider characters modulo the prime at infinity.
\subsection{Equivalence relation}
Let $\ell$ be a non-negative integer and $M$ be a non-zero polynomial in $\FF_q[T]$. We define an equivalence relation $R_{\ell,M}$ on $\mathcal{M}$ by saying that $A \equiv B \bmod R_{\ell,M}$ if and only if $A$ and $B$ have the same first $\ell$ next-to-leading coefficients and $A \equiv B \bmod M$. We adopt throughout the following convention: the $j$th next-to-leading coefficient of a polynomial $f(T) \in \mathcal{M}$ with $j > \deg (f)$ is considered to be $0$. 
It may be shown that there is a well-defined quotient monoid $\mathcal{M}/R_{\ell,M}$, where multiplication is the usual polynomial multiplication. An element of $\mathcal{M}$ is invertible modulo $R_{\ell,M}$ if and only if it is coprime to $M$. The units of $\mathcal{M}/R_{\ell,M}$ form an abelian group, having as identity element the equivalence class of the polynomial $1$. We denote this unit group by $\left( \mathcal{M} / R_{\ell,M}\right)^{\times}$. It may be shown that
\begin{equation}
\left| \left(\mathcal{M} / R_{\ell,M}\right)^{\times} \right| = q^{\ell} \phi(M).
\end{equation}
\subsubsection{Characters}
For every character $\chi$ of the finite abelian group $\left( \mathcal{M} / R_{\ell,M}\right)^{\times}$, we define $\chi^{\dagger}$ with domain $\mathcal{M}$ as follows. If $A$ is invertible modulo $R_{\ell,M}$ and if $\mathfrak{c}$ is the equivalence class of $A$, then $\chi^{\dagger}(A)=\chi(\mathfrak{c})$. If $A$ is not invertible, then $	\chi^{\dagger}(A)=0$.

The set of functions $\chi^{\dagger}$ defined in this way are called the characters of the relation $R_{\ell,M}$, or sometimes ``characters modulo $R_{\ell,M}$''. We shall abuse language somewhat and write $\chi$ instead of $\chi^{\dagger}$ to indicate a character of the relation $R_{\ell,M}$ derived from the character $\chi$ of the group $\left( \mathcal{M} / R_{\ell,M}\right)^{\times}$. Thus we write $\chi_0$ for the character of $R_{\ell,M}$ which has the value $1$ when $A$ is invertible and the value $0$ otherwise. We denote by $G(R_{\ell,M})$ the set $\{ \chi^{\dagger} : \chi \in \widehat{\left( \mathcal{M} / R_{\ell,M}\right)^{\times}} \}$. 

A set of polynomials in $\mathcal{M}$ is called a representative set modulo $R_{\ell,M}$ if the set contains one and only one polynomial from each equivalence class of $R_{\ell,M}$. If $\chi_{1},\chi_{2} \in G(R_{\ell,M})$, then
\begin{equation}\label{ortho1}
\frac{1}{q^{\ell}\phi(M)} \sum_{F} \chi_1(F) \overline{\chi_2}(F) = \begin{cases} 0 & \text{if }\chi_1 \neq \chi_2 ,\\ 1 & \text{if }\chi_1 = \chi_2 ,\end{cases}
\end{equation}
$F$ running through a representative set modulo $R_{\ell,M}$. If $n \ge \ell + \deg (M)$, then $\mathcal{M}_{n}$ is a disjoint union of $q^{n-\ell-\deg (M)}$ representative sets, and a set of polynomials on which $\chi\in G(R_{\ell,M})$ vanishes. Thus, applying \eqref{ortho1} with $\chi_2=\chi_0$, we obtain that for all $n \ge \ell+\deg(M)$,
\begin{equation}\label{orthouse}
\frac{1}{q^{n-\deg( M)}\phi(M)} \sum_{F \in \mathcal{M}_{n}} \chi(F) = \begin{cases} 0 & \text{if }\chi \neq \chi_0 ,\\ 1 & \text{if }\chi = \chi_0. \end{cases}
\end{equation}
We also have, for all $A,B \in \mathcal{M}$ coprime to $M$,             
\begin{equation}\label{ortho2}                        
\frac{1}{q^{\ell}\phi(M)}\sum_{\chi \in G(R_{\ell,M})} \chi(A)\overline{\chi}(B) = \begin{cases} 1 & \text{if }A\equiv B \bmod R_{\ell,M}, \\ 0 & \text{otherwise}. \end{cases}
\end{equation}                                     We also call the elements of $G(R_{\ell,M})$ ``generalized arithmetic progression characters," because for any $A \in \mathcal{M}_{n}$, $\chi \in G(R_{\ell,M})$ is constant on the set
\begin{equation*}
\{ f \in \mathcal{M}_{n} : f  \equiv A \bmod R_{\ell,M}\} = \{ f\in \FF_q[T] : f \equiv A \bmod M \} \cap \{ f \in \mathcal{M}_{n} : \deg(f-A) \le n-\ell-1 \}
\end{equation*}
which is an intersection of an arithmetic progression and a short interval. We set, for future use,
\begin{equation*}
\GAP(A;\ell,M) := \{ f \in \mathcal{M}_{\deg(A)} : f  \equiv A \bmod R_{\ell,M}\}.
\end{equation*}
If $\chi \in G(R_{\ell,1})$, we say that $\chi$ is a short interval character of $\ell$ coefficients, and if $\chi \in G(R_{0,M})$, we say that $\chi$ is a Dirichlet character modulo $M$. Every element of $G(R_{\ell,M})$ is a product of an element from $G(R_{\ell,1})$ with an element from $G(R_{0,M})$.  

\subsubsection{\texorpdfstring{$L$}{L}-Functions}\label{sec:lfunc}
Let $\chi \in G(R_{\ell,M})$. The $L$-function of $\chi$ is the following series in $u$:
\begin{equation}
L(u,\chi) = \sum_{f \in \mathcal{M}} \chi(f)u^{\deg (f)},
\end{equation}
which also admits the Euler product
\begin{equation}\label{eulerlchi}
L(u,\chi) = \prod_{P \in \mathcal{P}} (1-\chi(P)u^{\deg (P)})^{-1}.
\end{equation}
If $\chi$ is the trivial character $\chi_0$ of $G(R_{\ell,M})$, then
\begin{equation}
L(u,\chi) = \frac{\prod_{P \mid M} (1-u^{\deg (P)})}{1-qu}.
\end{equation}
Otherwise, the orthogonality relation \eqref{orthouse} implies that $L(u,\chi)$ is a polynomial in $u$ of degree at most
\begin{equation}\label{eq:degbound}
\deg L(u,\chi) \le \ell + \deg (M)-1.
\end{equation}
The first one to realize that Weil's proof of the Riemann Hypothesis for Function Fields \cite[Theorem~6,~p.~134]{weil1974} implies the Riemann Hypothesis for the $L$-functions of $\chi \in G(R_{\ell,M})$ was Rhin \cite[Chapitre~2]{rhin1972} in his thesis (cf. \cite[Theorem~5.6]{effinger1991} and the discussion following it). Hence we know that if we factor $L(u,\chi)$ as
\begin{equation}\label{eq:roots}
L(u,\chi) = \prod_{i=1}^{\deg L(u,\chi)} (1-\gamma_i(\chi)u),
\end{equation}
then for any $i$, 
\begin{equation}\label{eq:rh}
\left|\gamma_i(\chi) \right| \in \{1, \sqrt{q}\}.
\end{equation}
The following consequence of \eqref{eq:rh} is a basic ingredient in our proofs.
\begin{lem}\label{lem:PrimeBound}
Let $M$ be a non-zero polynomial in $\FF_q[T]$ and $\ell$ a non-negative integer. Let $\chi \in G(R_{\ell,M})$ and $n \ge 1$. Then
\begin{equation}
\left| \sum_{P \in \mathcal{P}_{n}} \chi(P) \right| \le \begin{cases} \min\{\frac{q^{\frac{n}{2}}}{n}(\ell+\deg (M)+1), \frac{q^n}{n} \}& \mbox{if $\chi \neq \chi_0$,}\\ \frac{q^n}{n} & \mbox{otherwise.} \end{cases}
\end{equation}
\end{lem}
 \begin{proof}
The bound $\left| \sum_{P \in \mathcal{P}_{n}} \chi(P) \right|  \le q^n/n$ follows from the bound $|\mathcal{P}_{n}| \le q^n/n$ \cite[Proposition~2.1]{rosen2002}. For $\chi \neq \chi_0$, we obtain the additional bound as follows. We equate \eqref{eulerlchi} with \eqref{eq:roots} and take logarithmetic derivatives to obtain
\begin{equation}\label{eq:vonmangsum}
|\sum_{f \in \mathcal{M}_{n}} \Lambda(f) \chi(f)| = |\sum_{i=1}^{\deg L(u,\chi)} \gamma_i(\chi)^n| \le (\ell+\deg(M)-1)q^{\frac{n}{2}},
\end{equation}
where the last passage used \eqref{eq:degbound} and \eqref{eq:rh}. We can split \eqref{eq:vonmangsum} into the contribution of primes of degree $n$ and proper prime powers:
 \begin{equation}\label{vonmangbound2}
\Big| n\sum_{f \in \mathcal{P}_{n}} \chi(f) + \sum_{d \mid n, \, d \neq 1} \frac{n}{d} \sum_{f \in \mathcal{P}_{\frac{n}{d}}} \chi^d(f) \Big| \le (\ell+\deg (M) - 1) q^{\frac{n}{2}}.
\end{equation}
As
\begin{equation}
\begin{split}
\left| \sum_{d \mid n, \, d \neq 1} \frac{n}{d}\sum_{f \in \mathcal{P}_{\frac{n}{d}}} \chi^d(f) \right| & \le \sum_{d \mid n, \, d \neq 1} \frac{n}{d} |\mathcal{P}_{n/d}| \le \sum_{d \mid n, \, d \neq 1} q^{n/d} \le 2q^{\frac{n}{2}},
\end{split}
\end{equation}
we obtain from \eqref{vonmangbound2} and the triangle inequality that
\begin{equation}
|n \sum_{f \in \mathcal{P}_{n}} \chi(f)| \le (\ell+\deg (M) +1)q^{\frac{n}{2}}.
\end{equation}
After dividing by $n$, the lemma is established.
 \end{proof}

\subsection{Sums over generalized arithmetic progressions and their variance}
For an arithmetic function $\alpha\colon  \mathcal{M} \to \mathbb{C}$ and $n \ge 1$, define
\begin{equation}\label{eq:snalpha}
S(n,\alpha) := \sum_{f\in \mathcal{M}_{n}} \alpha(f).
\end{equation}
The following lemma expresses sums over generalized arithmetic progressions, and the variance of such sums, as sums over characters in $G(R_{\ell,M})$. Special cases of this lemma appeared in a paper of Keating and Rudnick \cite{keating2016}. 
\begin{lem}\label{lemshort}
Let $\ell$ be a non-negative integer and $M \in \FF_q[T]$ be a non-zero polynomial. Let $n \ge \ell + \deg(M)$. Let $A \in (\mathcal{M} /R_{\ell,M})^{\times}$ and define $f_A \in \mathcal{M}$ to be some polynomial of degree $n$ in the equivalence class of $A$ modulo $R_{\ell,M}$. Then the following hold.
\begin{enumerate}
\item We have, for all $g \in  \mathcal{M}_{n;M}$,
\begin{equation}
\Ind_{g \in \GAP(f_A;\ell,M)} = \frac{\sum_{\chi \in G(R_{\ell,M})} \overline{\chi}(A) \chi(g)}{\phi(M)q^{\ell}}.
\end{equation}
\item For any arithmetic function $\alpha \colon \mathcal{M} \to \mathbb{C}$, we have
\begin{equation}\label{valuesum}
\begin{split}
\sum_{g \in \GAP(f_A;\ell,M)} \alpha(g) &= \frac{\sum_{\chi \in G(R_{\ell,M})} \overline{\chi}(A) S(n,\alpha \cdot \chi )}{\phi(M)q^{\ell}}\\
&=q^{n-\ell-\deg(M)} \langle \alpha \rangle_{\mathcal{M}_{n;M}}+ \frac{\sum_{\chi_0 \neq \chi \in G(R_{\ell,M})} \overline{\chi}(A) S(n,\alpha \cdot \chi)}{q^{\ell}\phi(M)}.
\end{split}
\end{equation}
\item For any arithmetic function $\alpha \colon \mathcal{M} \to \mathbb{C}$, the variance of $\{ \sum_{g \in GAP(f;n,\ell;M)} \alpha(g)\}_{f \in \mathcal{M}_{n;M}}$, defined as
\begin{equation}
\mathrm{Var}_M(\alpha;n,n-\ell-1):=\frac{1}{q^{n-\deg(M)}\phi(M)} \sum_{f \in \mathcal{M}_{n;M}} \left| \sum_{g \in \GAP(f;\ell,M)} \alpha(g) - q^{n-\ell-\deg(M)} \langle \alpha \rangle_{\mathcal{M}_{n;M}} \right|^2,
\end{equation}
may be expressed as
\begin{equation}\label{varform}
\mathrm{Var}_M(\alpha;n,n-\ell-1)= \frac{\sum_{\chi_0 \neq \chi \in G(R_{\ell,M})} \left| S(n,\alpha \cdot\chi)\right|^2 }{q^{2\ell }\phi^2(M)}.
\end{equation}
\end{enumerate}
\end{lem}
\begin{proof}
The first part of the lemma is a restatement of the orthogonality relation \eqref{ortho2}. For the second part of the lemma, we observe that $\sum_{g \in \GAP(f_A;\ell,M)} \alpha(g) = \sum_{g \in \mathcal{M}_{n;M}} \alpha(g) \cdot \Ind_{g \in \GAP(f_A;\ell,M)}$, and now we apply the first part of the lemma and interchange the order of summation. Note that
\begin{equation}\label{meanval}
\frac{S(n,\alpha \cdot \chi_0)}{\phi(M)q^{\ell}}= q^{n-\ell-\deg(M)} \langle \alpha \rangle_{\mathcal{M}_{n;M}}.
\end{equation}
To prove the last part of the lemma, we use \eqref{meanval} and \eqref{valuesum} as follows:
\begin{equation}\label{varproof}
\begin{split}
\mathrm{Var}_M(\alpha;n,n-\ell-1)&= \frac{1}{q^{\ell}\phi(M)}  \sum_{(A \in \mathcal{M} / R_{\ell,M})^{\times}} \left| \sum_{g \in \GAP(f_A;\ell,M)} \alpha(g) - \frac{S(n,\alpha\cdot \chi_0)}{\phi(M)q^{\ell}} \right|^2 \\
 &= \frac{1}{q^{\ell}\phi(M)} \sum_{(A \in \mathcal{M} / R_{\ell,M})^{\times}} \left| \frac{\sum_{\chi_0 \neq \chi \in G(R_{\ell,M})} \overline{\chi}(A) S(n,\alpha \cdot \chi)}{\phi(M)q^{\ell}}  \right|^2 \\
 &= \frac{1}{q^{3\ell}\phi^3(M)} \sum_{(A \in \mathcal{M} / R_{\ell,M})^{\times}}\sum_{\chi_1,\chi_2 \in G(R_{\ell,M}) \setminus \{ \chi_0\}} \overline{\chi_1}(A) S(n,\alpha \cdot \chi_1) \chi_2(A) \overline{S(n,\alpha \cdot \chi_2)}.  
\end{split}
\end{equation}
Interchanging the order of summation and applying \eqref{ortho1}, we conclude the proof.
\end{proof}
\section{Proofs of Theorems~\ref{cor:ForFunc}, \ref{cor:ForFuncAP} and \ref{thm:var}}
Here we prove the main results of the paper using Lemma~\ref{lemshort} and Theorem~\ref{thm:expothm}. By \eqref{valuesum} with $M=1$ and $\ell=n-h-1$, we obtain that
\begin{equation}
\left| \langle f \rangle_{I(f_0,h)} - \langle f \rangle_{\mathcal{M}_{n}} \right| \le \frac{\sum_{\chi_0 \neq \chi \in G(R_{n-h-1,1})} \left| S(n,\alpha \cdot \chi)\right| }{q^n} .
\end{equation}
Plugging \eqref{eq:upper2} in the sum above, we conclude the proof of Theorem~\ref{cor:ForFunc}. By \eqref{valuesum} with $\ell=0$, we obtain that
\begin{equation}
\left| \langle f \rangle_{\AP(f_0,M)} - \langle f \rangle_{\mathcal{M}_{n;M}} \right| \le \frac{\sum_{\chi_0 \neq \chi \in G(R_{0,M})} \left| S(n,\alpha \cdot \chi)\right| }{q^{n-\deg(M)}\phi(M)} .
\end{equation}
Plugging \eqref{eq:upper2} in the last sum, we conclude the proof of Theorem~\ref{cor:ForFuncAP}. By \eqref{varform} with $M=1$ and $\ell=n-h-1$, we obtain that
\begin{equation}
\mathrm{Var}(\alpha;n,h)=   \frac{\sum_{\chi_0 \neq \chi \in G(R_{n-h-1,1})} \left| S(n,\alpha \cdot\chi)\right|^2 }{q^{2(n-h-1)}}.
\end{equation}
Plugging \eqref{eq:upper2} in the sum above, \eqref{eq:upper_functionvar} is established, and \eqref{eq:upper_functionvarap} is proved similarly, by applying \eqref{varform} with $\ell=0$. This concludes the proof of Theorem~\ref{thm:var}.
\section{Preparation for proof of Theorem~\ref{thm:expothm}}
Here and in subsequent sections, we shall use the notation $[u^n]f(u)$ for the coefficient of $u^n$ in a power series $f$. We also write $\exp(\bullet)$ for $e^{\bullet}$.
\subsection{Multiplicativity of character sums}
Given $d \ge 1$ and a factorization type
\begin{equation}
\FactType =\{ (d_i, e_i): 1 \le i \le k\}\in \Omega,
\end{equation}
we denote by $\FactType(d) \subseteq \FactType$ the factorization type $\{ (d_i, e_i): 1 \le i \le k,\,  d_i = d \}$. By definition, $\FactType$ is the disjoint of union of the $\FactType(d)$-s. Let $\Ind_{\FactType}$ be the indicator function of polynomials $f$ with $\FactType_f=\FactType$. The following lemma shows that the character sums $S(n,\chi \cdot \Ind_{\FactType})$ (recall \eqref{eq:snalpha}) enjoy a multiplicative property.
\begin{lem}\label{lem:multi}
Let $\FactType \in \Omega$ with $|\FactType|=n$. Let $\chi$ be a Hayes character. Then
\begin{equation}\label{eq:snmult}
S(n,\chi \cdot \Ind_{\FactType}) = \prod_{d=1}^{n} S(|\FactType(d)|,\chi \cdot \Ind_{\FactType(d)}).
\end{equation}
\end{lem}
\begin{proof}
Each $f$ with $\FactType_f=\FactType$ can be  written uniquely as $f=\prod_{d=1}^{n} f_d$ where $f_d$ is divisible only by primes of degree $d$.  We then have $\FactType_{f_d}=\FactType(d)$, and the lemma follows by expanding the right-hand side of \eqref{eq:snmult}.
\end{proof}

\subsection{Symmetric function theory}
A partition of size $n$ is a finite (possibly empty) non-increasing sequence of positive integers that sum to $n$. The length of a partition $\lambda=(\lambda_1, \lambda_2, \ldots , \lambda_k)$ is the number of its elements and is denoted by $\ell(\lambda):=k$. We write $\lambda \vdash n$ to indicate that $\lambda$ sums to $n$. The empty partition is of size and length $0$. We denote by $\YY$ the set of all partitions.

An important class of symmetric polynomials is the \emph{monomial symmetric polynomials}. Given a partition $\lambda=(\lambda_1, \lambda_2, \ldots,   \lambda_k)$ and variables $X_1,\ldots, X_k$, the monomial symmetric polynomial $m_{\lambda}(X_1,\ldots,X_k)$ is the symmetric polynomial
\begin{equation}
m_{\lambda}(X_1,\ldots, X_k) := \sum_{\exists \pi \in S_k : (\lambda'_1,\ldots,\lambda'_k) =(\lambda_{\pi(1)},\ldots, \lambda_{\pi(k)})} \prod_{i=1}^{k} X^{\lambda'_i}_{i} \in \ZZ[X_1,\ldots,X_k],
\end{equation}
where the sum is over the \emph{distinct} permutations of $\lambda$. It is useful to extend $m_{\lambda}$ to the case of a general number of variables $X_1,\ldots,X_n$. If $n<k$, we define $m_{\lambda}(X_i: 1 \le i \le n)$ to be zero. If $n>k$, we set $\lambda_j=0$ for $j=k+1,\ldots,n$ and define
\begin{equation}
m_{\lambda}(X_i : 1 \le i \le n) := \sum_{\exists \pi \in S_n : (\lambda'_1,\ldots,\lambda'_n) =(\lambda_{\pi(1)},\ldots,\lambda_{\pi(n)})} \prod_{i=1}^{n} X^{\lambda'_i}_{i}\in \ZZ[X_1,\ldots,X_n],
\end{equation}
where the sum is over the \emph{distinct} permutations of $\lambda$ followed by $n-k$ zeros. In particular, $m_{\lambda}$ is the elementary symmetric polynomial $e_{k}$ if $\lambda=(1,1,\ldots,1)$ with $k$ ones. 

The following lemma expresses character sums, of the form appearing in the right-hand side of \eqref{eq:snmult}, as an evaluation of a monomial symmetric polynomial.
\begin{lem}\label{lem:monomial}
Let $\FactType=\{(d,e_i): 1 \le i \le k\} \in \Omega$. Let $\lambda\in \YY$ be the partition whose parts are $\{ e_i\}_{i=1}^{k}$ in non-increasing order. Let $\chi$ be a Hayes character. We have
\begin{equation}
S(|\FactType|, \chi \cdot \Ind_{\FactType}) = m_{\lambda}(\chi(P): P \in \mathcal{P}_{d}).
\end{equation}
\end{lem}
\begin{proof}
A polynomial $f$ with $\FactType_f=\FactType$ is necessarily given by a product $\prod_{i=1}^{k} P_i^{e_i}$ where $P_i$ are distinct elements from $\mathcal{P}_d$. Equivalently, $f$ may be expressed as  $\prod_{P \in \mathcal{P}_d} P^{e(P)}$ where the multiset $\{e(P) : P \in \mathcal{P}_d\}$ is equal to
\begin{equation}
E: = \{ e_i : 1 \le i \le k\} \cup \{ 0 : 1 \le i \le |\mathcal{P}_d|-k\}.
\end{equation}
Moreover, by unique factorization, this form is unique. Thus,
\begin{equation}
\begin{split}
S(|\FactType|, \chi \cdot \Ind_{\FactType}) &= \sum_{f \in \mathcal{M}_{|\FactType|} : \FactType_f = \FactType} \chi(f) = \sum_{\substack{e\colon\mathcal{P}_d \to \mathbb{N}_{\ge 0} \\ \{ e(P) : P \in \mathcal{P}_d\} = E}} \chi\left(\prod_{P \in \mathcal{P}_d} P^{e(P)} \right)\\
&= \sum_{\substack{e\colon\mathcal{P}_d \to \mathbb{N}_{\ge 0} \\ \{ e(P) : P \in \mathcal{P}_d\} = E}} \prod_{P \in \mathcal{P}_d} \big(\chi(P) \big)^{e(P)},
\end{split}
\end{equation}
which is just $m_{\lambda}$ evaluated at $\{\chi(P): P \in \mathcal{P}_d\}$, as needed.
\end{proof}
Another class of symmetric polynomials is the \emph{power sum symmetric polynomials}. Given a positive integer $r$, the power sum symmetric polynomial $p_r(X_i: 1 \le i \le n)$ is the symmetric polynomial
\begin{equation}
p_{r}(X_i: 1 \le i \le n):= \sum_{i=1}^{n} X_i^r \in \ZZ[X_1,\ldots,X_n].
\end{equation}
More generally, given a partition $\lambda=(\lambda_1, \lambda_2, \ldots,  \lambda_k)$, then the power sum symmetric polynomial $p_{\lambda}(X_i: 1 \le i \le n)$ is the symmetric polynomial
\begin{equation}
p_{\lambda}(X_i: 1 \le i \le n) := \prod_{i=1}^{k} p_{\lambda_i}(X_i: 1 \le i \le n).
\end{equation}
A basic result in symmetric function theory says that whenever $m \ge n$, $\{ m_{\lambda}(X_i: 1 \le i \le m) \}_{\lambda \vdash n}$ and $\{ p_{\lambda}(X_i: 1 \le i \le m) \}_{\lambda \vdash n}$ are both bases for homogeneous symmetric polynomials of degree $n$ with rational coefficients. In particular, $m_{\lambda}(X_i: 1 \le i \le m)$ can be expressed uniquely as a linear combination of the symmetric polynomials $p_{\mu}$ for partitions $\mu$ of size $n$, that is, there are unique coefficients $c_{\lambda,\mu}\in \mathbb{Q}$ such that
\begin{equation}\label{eq:basechange}
m_{\lambda}(X_i: 1 \le i \le m) = \sum_{\mu \vdash n} c_{\lambda,\mu} p_{\mu}(X_i: 1 \le i \le m)
\end{equation}
for all $m \ge n$ ($c_{\lambda,\mu}$ are independent of $m$). E\u{g}ecio\u{g}lu and Remmel \cite[pp.~107--111]{ege1991} gave a combinatorial interpretation of $c_{\lambda,\mu}$ which we now describe. We begin with their definition of $\lambda$-brick tabloids.

Let $\lambda=(\lambda_1, \lambda_2, \ldots,  \lambda_k)$ and $\mu=(\mu_1, \mu_2, \ldots, \mu_r)$ be two partitions. Recall that the \emph{Young diagram} $Y_{\mu}$ is the diagram which consists of left justified rows of squares of lengths $\mu_1,\mu_2,\ldots,\mu_k$ reading from top to bottom. For instance, if $\mu=(4, 3)$ then $Y_{\mu}$ is given by
\begin{equation}
Y_{\mu}=  \ydiagram{4,3}\, .
\end{equation}
 A \emph{$\lambda$-brick tabloid $T$ of shape $\mu$} is a filling of $Y_{\mu}$ with bricks $b_1,\ldots,b_k$ of lengths $\lambda_1,\ldots, \lambda_k$, respectively, such that 
\begin{enumerate}
\item each brick $b_i$ covers exactly $\lambda_i$ squares of $Y_{\mu}$ all of which lie in a single row of $Y_{\mu}$,
\item no two brick overlap.
\end{enumerate}
For example, if $\lambda=(3, 2, 1, 1)$ and $\mu=(4, 3)$, then we must cover $Y_{\mu}$ with the bricks 
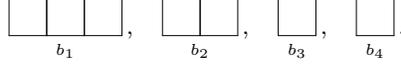
\begin{equation}
\substack{\ydiagram{3} \vspace{2pt} \\  b_1}\,, \quad \substack{\ydiagram{2} \vspace{2pt}\\ b_2}\,, \quad \substack{ \ydiagram{1} \vspace{2pt}\\ b_3}\,, \quad  \substack{\ydiagram{1} \vspace{2pt}\\b_4} \, .
\end{equation}
Here, bricks of the same size are indistinguishable. There are in total seven $\lambda$-brick tabloids of shape $\mu$, given in Figure~\ref{fig:tabs}.

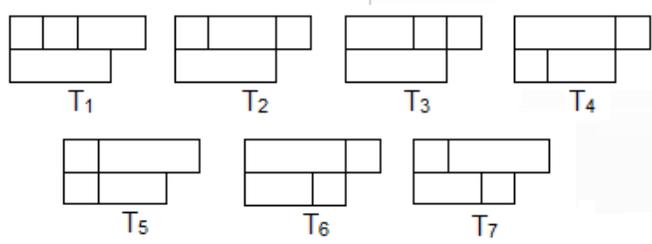
\begin{figure}
	\centering
\begin{tikzpicture}[scale=0.50]
\draw[black] (0,0)rectangle (1,1);
\draw[black] (1,0)rectangle (2,1);
\draw[black] (2,0)rectangle (4,1);
\draw[black] (0,-1)rectangle (3,0);

\draw[black] (5,0)rectangle (6,1);
\draw[black] (6,0)rectangle (8,1);
\draw[black] (8,0)rectangle (9,1);
\draw[black] (5,-1)rectangle (8,0);

\draw[black] (10,0)rectangle (12,1);
\draw[black] (12,0)rectangle (13,1);
\draw[black] (13,0)rectangle (14,1);
\draw[black] (10,-1)rectangle (13,0);

\draw[black] (15,0)rectangle (18,1);
\draw[black] (18,0)rectangle (19,1);
\draw[black] (15,-1)rectangle (16,0);
\draw[black] (16,-1)rectangle (18,0);

\node[below] at (2,-1.2) {$T_1$};
\node[below] at (7,-1.2) {$T_2$};
\node[below] at (12,-1.2) {$T_3$};
\node[below] at (17,-1.2) {$T_4$};

\draw[black] (2,-3.2)rectangle (3,-2.2);
\draw[black] (3,-3.2)rectangle (6,-2.2);
\draw[black] (2,-4.2)rectangle (3,-3.2);
\draw[black] (3,-4.2)rectangle (5,-3.2);

\draw[black] (7,-3.2)rectangle (10,-2.2);
\draw[black] (10,-3.2)rectangle (11,-2.2);
\draw[black] (7,-4.2)rectangle (9,-3.2);
\draw[black] (9,-4.2)rectangle (10,-3.2);

\draw[black] (12,-3.2)rectangle (13,-2.2);
\draw[black] (13,-3.2)rectangle (16,-2.2);
\draw[black] (12,-4.2)rectangle (14,-3.2);
\draw[black] (14,-4.2)rectangle (15,-3.2);

\node[below] at (4,-4.4) {$T_5$};
\node[below] at (9,-4.4) {$T_6$};
\node[below] at (14,-4.4) {$T_7$};
\end{tikzpicture}
\caption{$(3,2,1,1)$-brick tabloids of shape $(4,3)$} \label{fig:tabs}

\end{figure}
We let $B_{\lambda, \mu}$ denote the set of $\lambda$-brick tabloids of shape $\mu$. We define a weight $w(T)$ for each $\lambda$-brick tabloid $T\in B_{\lambda,\mu}$ by
\begin{equation}
w(T)=\prod_{b \in T}w_T(b),
\end{equation}
where for each brick $b$ in $T$, $|b|$ denotes the length of $b$ and 
\begin{equation}
w_T(b)=\begin{cases} 1 & \mbox{if $b$ is not at the end of a row in $T$,} \\ |b| &\mbox{if $b$ is at the end of a row in $T$.} \end{cases}
\end{equation}
Thus $w(T)$ is the product of the lengths of the rightmost bricks in $T$. For example,
for the seven $(3, 2, 1, 1)$-brick tabloids of shape $(4, 3)$ given in Figure \ref{fig:tabs}, the weights are computed to be $w(T_1)= 6$, $w(T_2)= 3$, $w(T_3)= 3$, $w(T_4)= 2$, $w(T_5)= 6$, $w(T_6)=1$ and $w(T_7)= 3$. We let
\begin{equation}
w(B_{\lambda,\mu}):=\sum_{T \in B_{\lambda,\mu}} w(T).
\end{equation}
E\u{g}ecio\u{g}lu and Remmel \cite[Rel.~(11),~p.~111]{ege1991} proved that for partitions $\lambda, \mu \vdash n$, we have
\begin{equation}\label{eq:cdef}
c_{\lambda,\mu} = (-1)^{\ell(\lambda)-\ell(\mu)}w(B_{\lambda,\mu}) \Prob_{\mu},
\end{equation}
where
\begin{equation}
\Prob_{\mu}:= \Prob_{\pi \in S_n}(\pi \text{ has cycle type }\mu).
\end{equation}
Here $\Prob_{\pi \in S_n}$ is the uniform probability measure on the symmetric group $S_n$, and we say that $\pi$ has a cycle type $(\mu_1, \mu_2, \ldots, \mu_r)$ if the cycle sizes of $\pi$ are given by $\mu_1, \ldots, \mu_r$.

\begin{lem}\label{lem:bsum}
Let $n$ and $k$ be positive integers. Let $\mu \vdash n$. We have
\begin{equation}\label{eq:sumomegas}
\sum_{\substack{\lambda \vdash n \\ \ell(\lambda) \le k}} w(B_{\lambda,\mu}) \le \sum_{i=0}^{k}\binom{n}{i}.
\end{equation}
\end{lem}
\begin{proof}
Write $\mu$ as $(\mu_1,\ldots,\mu_r)$. A $\lambda$-brick tabloid of shape $\mu$ determines the partition $\lambda$ uniquely. Indeed, $\lambda$ can be recovered by reading the lengths of the bricks in each row of the tabloid. Thus, the set $\cup_{\lambda\vdash n, \, \ell(\lambda) \le k}B_{\lambda,\mu}$ may be identified with a sequence $\{ b_i\}_{i=1}^{r}$ of positive integers with $\sum_{i=1}^{r} b_i \le k$, and a double sequence $\{ a_{i,j} \}_{ 1 \le i \le r,\,1 \le j \le b_i}$ of positive integers with $\sum_{j=1}^{b_i} a_{i,j}=\mu_i$ for each $i$ as follows. The number $b_i$ is set to be the number of blocks in the $i$th topmost row of the tabloid, and the number $a_{i,j}$ is set to be the length of the $j$th leftmost brick in the $i$-topmost row. Under this identification, $w(B_{\lambda,\mu})$ is given by the product $\prod_{i=1}^{r}a_{i,b_i}$, and it follows that for any $t \ge 0$, we have
\begin{equation}\label{eq:paramblambda}
\sum_{\substack{\lambda \vdash n \\ \ell(\lambda) = t}} w(B_{\lambda,\mu}) = \sum_{\substack{b_1,\ldots,b_{r} \ge 1\\ b_1+\ldots + b_r = t}} \sum_{\substack{\forall 1 \le i \le r:\\ a_{i,1},\ldots, a_{i,b_i} \ge 1 \\ \sum_{j=1}^{b_i} a_{i,j}=\mu_i}} \prod_{i=1}^{r} a_{i,b_i}.
\end{equation}
Consider the generating function
\begin{equation}
B(u) := \sum_{\lambda \vdash n} \omega(B_{\lambda,\mu}) u^{\ell(\lambda)}.
\end{equation}
Letting $c(n_1,n_2, n_3)$ be the number of solutions to $x_1+x_2+\ldots + x_{n_1}= n_3$ with $x_{n_1}=n_2$ and $x_i \ge 1$, it follows from \eqref{eq:paramblambda} that 
\begin{equation}
B(u)=\sum_{\substack{\forall 1 \le i \le r:\\b_i, y_i \ge 1}} \prod_{i=1}^{r} c(b_i,y_i,\mu_i) y_i u^{b_i} = \prod_{i=1}^{r} (\sum_{b_i,y_i \ge 1} c(b_i,y_i,\mu_i)y_i u^{b_i}).
\end{equation} 
As $c(n_1,n_2,n_3)$ is also the number of solutions to $x_1+\ldots + x_{n_1-1}=n_3-n_2$ in positive integers, a standard combinatorial result says that
\begin{equation}
c(n_1,n_2,n_3) = \begin{cases} \binom{n_3-n_2-1}{n_1-2} & \mbox{if $n_3 \ge n_2+n_1-1$, $n_1 \ge 2$,}\\ \Ind_{n_2=n_3} & \mbox{if $n_1=1$,} \\ 0 & \mbox{otherwise,} \end{cases}
\end{equation}
so that
\begin{equation}
\begin{split}
B(u) &= \prod_{i=1}^{r} (\mu_i u+ \sum_{y=1}^{\mu_i-1} \sum_{b=2}^{\mu_i-y+1} \binom{\mu_i-y-1}{b-2} yu^{b})\\
&= \prod_{i=1}^{r} (\mu_i u+ \sum_{y=1}^{\mu_i-1} yu^2 (1+u)^{\mu_i-y-1})\\
&=\prod_{i=1}^{r} ((1+u)^{\mu_i}-1),
\end{split}
\end{equation} 
where in the last passage we made use of the identity $\sum_{i=1}^{d}ix^i =x(dx^{d+1}-(d+1)x^d+1)/(x-1)^2$ with $d=\mu_i-1$ and $x=1/(1+u)$. As the left-hand side of \eqref{eq:sumomegas} is the sum of the first $k+1$ coefficients of $B(u)$, and they are bounded from above by the corresponding coefficients of $\prod_{i=1}^{r}(1+u)^{\mu_i} = (1+u)^n = \sum_{i=0}^{n} \binom{n}{i}u^i$, the proof is concluded.
\end{proof}

\subsection{Permutation statistics}
We denote the expectation of a function $f\colon S_n\to \mathbb{R}$ with respect to the uniform probability measure on $S_n$ by $\mathrm{E}_{\pi \in S_n}f(\pi)$. We denote by $\ell(\pi)$ the number of cycles in a permutation $\pi$. 
\begin{lem}\label{lem:meanz}
Let $n \ge 1$, $m \ge 2$ be positive integers. Let $z_1,z_2 \in \mathbb{C}$. Define the following function on $S_n$:
\begin{equation}\label{eq:deff}
f(\pi) = \prod_{C \in \pi,\, m \nmid \left| C \right|} z_1 \prod_{C \in \pi, \, m \mid \left| C \right|} z_2
\end{equation}
where the product is over the disjoint cycles of $\pi$. Then
\begin{equation}
\mathrm{E}_{\pi \in S_n} f(\pi)= [u^n](1-u)^{-z_1}(1-u^m)^{(-z_2+z_1)/m}.
\end{equation}
\end{lem}
\begin{proof}
The exponential formula for permutations \cite[Corollary~5.1.9]{stanley1999enumerative} states the following. Given a function $g\colon \mathbb{N}\to \mathbb{C}$, we construct a corresponding function on permutations (on arbitrary number of elements) as follows:
\begin{equation}
G(\pi) = \prod_{C \in \pi} g(|C|),
\end{equation}
where the product is over the disjoint cycles of $\pi$. We then have the following identity of formal power series:
\begin{equation}
1+\sum_{i \ge 1} \left(\mathrm{E}_{\pi \in S_i}  G(\pi)\right) u^i = \exp(\sum_{j \ge 1} \frac{g(j)}{j}u^j).
\end{equation}
Applying the identity with
\begin{equation}
g(j) = \begin{cases} z_1 & \mbox{if $m \nmid j$,} \\ z_2 & \mbox{otherwise,} \end{cases}
\end{equation}
we find that $G(\pi)=f(\pi)$ for every $\pi \in S_n$, and the lemma follows by a short computation.
\end{proof}
\subsection{Bounds on certain finite sums}
\begin{lem}\label{lem:bounds}
Let $x \ge 2$ be a real number and $n$ be a positive integer. Then
\begin{enumerate}
\item $\sum_{d_1 d_2 = n} 2^{d_1} x^{d_2} \le 8 x^n$.
\item If furthermore $x \ge 4$, $\sum_{d_1 d_2 = n,\, d_2<n} 2^{d_1} x^{d_2} \le 10 x^{n/2}$.
\end{enumerate}
\end{lem}
\begin{proof}
We begin with the first part. We may assume $n>1$. Consider the function $f(t)=2^{\frac{n}{t}}x^{t}$ on $[1,n]$. Its derivative is
\begin{equation}
f'(t) = f(t) (\log x - \frac{n \log 2}{t^2}),
\end{equation}
so that $f$ is either increasing from $1$ to $n$ (if $n\log 2 / \log x \le 1$), or decreasing from $1$ to $\sqrt{n \log 2/\log x}$ and increasing from $\sqrt{n \log 2/\log x}$ to $n$ (otherwise). As $f(1)=2^{n}x \le 2x^{n} = f(n)$, it follows that
\begin{equation}
\sum_{d_1 d_2 = n} 2^{d_1} x^{d_2} =\sum_{d_2 \mid n} f(d_2) \le 4x^n + (n-2) \max\{ f(2),f(n/2) \} \cdot \Ind_{n \text{ not a prime}}.
\end{equation}
If $n$ is a prime, we are done. Otherwise, it suffices to show that $(n-2)2^2 x^{n/2} \le 4x^n$ and that $(n-2)2^{n/2}x^2 \le 4x^n$. For $n =4,5$, these are easy to verify, and for larger $n$ they follow by induction. This establishes the first part of the lemma. The proof of the second part follows similar lines and is therefore omitted.
\end{proof}
The following variant of Lemma~\ref{lem:bounds} is also needed. We omit the similar proof.
\begin{lem}\label{lem:bounds2}
Let $x \in \{ \sqrt{2}, \sqrt{3}, 2, 3\}$. For any positive integer $n$, we have
\begin{equation}
\sum_{d_1 d_2 = n,\, d_1 \neq n} 2^{d_1} x^{d_2} \le 7 x^{n}.
\end{equation}
For $x=2,3$ we also have $\sum_{d_1 d_2 = n, \, 1<d_2<n} 2^{d_1} x^{d_2} + 1.4^n x \le 14 x^{n/2}$.
\end{lem}
\begin{lem}\label{lem:binomtoexp}
For any $n \ge 1$, we have $\sum_{i=0}^{3} \binom{n}{i} \le 7 \cdot 1.4^n$.
\end{lem}
\begin{proof}
For $n \le 9$, this is checked by a short computation. For $n \ge 10$, we have $\sum_{i=0}^{3} \binom{n}{i} = (n^3+5n+6)/6 \le 2n^3$, and an inductive argument shows that $2n^3 \le 7 \cdot 1.4^n$. 
\end{proof}

\subsection{Bounds on coefficients of a generating function}
\begin{lem}\label{lem:genfunc}
Let $t, r \ge 2$. Set $L:=\lfloor 2 \log_t r \rfloor$ and
\begin{equation}
Z(u):=\exp\Big(\sum_{1 \le k \le L} \frac{20t^k}{k} u^k + \sum_{k > L} 20(r+1) \frac{t^{\frac{k}{2}}}{k} u^k \Big).
\end{equation}
Then
\begin{equation}\label{eq:binombound}
\left| [u^n] Z(u) \right| \le t^{\frac{n}{2}} \binom{20(r+1)+n-1}{n}
\end{equation}
for all $n \ge 1$. If $r \ge \max\{200000, t^{\log^2 t} \}$, we have
\begin{equation}\label{eq:coolbound}
\left| [u^n] Z(u) \right| \le t^{\frac{n}{2}} t^{\frac{n \log \log r}{\log r}} \exp(140\frac{(r+1)}{(\log r)^2}t)
\end{equation}
for all $n \ge 1$. If $r = O(n)$, we have
\begin{equation}\label{eq:coolbound2}
\left| [u^n] Z(u) \right| \le t^{\frac{n}{2}+O_t(\frac{n \log \log (n+2)}{\log (n+2)})}.
\end{equation}

\end{lem}
\begin{proof}
We work with the modified function
\begin{equation}
\widetilde{Z}(u):=Z(u/\sqrt{t}),
\end{equation}
so that
\begin{equation}
[u^n] Z(u) = t^{n/2}[u^n]\widetilde{Z}(u).
\end{equation}
As $t^{k/2} \le r+1$ for $k \le L$, we have 
\begin{equation}
\left| [u^n] \widetilde{Z}(u) \right| \le \left| [u^n] \exp( \sum_{k \ge 1} 20(r+1)\frac{u^k}{k}) \right| =  [u^n] (1-u)^{-20(r+1)},
\end{equation}
which establishes \eqref{eq:binombound}. As $\widetilde{Z}$ has non-negative coefficients and radius of convergence $1$, we have
\begin{equation}
|[u^n] \widetilde{Z}(u) | \le \sum_{i \ge 0}R^{i-n}  [u^i]\widetilde{Z}(u)= \frac{\widetilde{Z}(R)}{R^n}
\end{equation}
for every $R\in (0,1)$. If $R\in (1.2/\sqrt{t},1)$, we can bound $\sum_{ 1 \le k \le L} 20t^{k/2}R^k/k$ from above by
\begin{equation}
\sum_{ 1 \le k \le L} \frac{20t^{k/2}}{k}R^k \le \frac{20(R\sqrt{t})^L}{1-1/(R\sqrt{t})} \le 120(r+1) R^L,
\end{equation}
and the sum $\sum_{k>L} 20(r+1) R^k/k$ by
\begin{equation}
\sum_{k>L} 20(r+1) \frac{R^k}{k} \le \frac{20(r+1)}{L+1} \frac{R^{L}}{1-R}.
\end{equation}
Thus, for every $R \in (1.2/\sqrt{t},1)$,
\begin{equation}\label{eq:zurbound}
|[u^n]\widetilde{Z}(u)| \le \exp\Big( 20(r+1)R^L (6 + \frac{1}{(L+1)(1-R)}) - n \log R\Big).    
\end{equation}
Assume $r \ge 200000$ and choose $R=t^{-\frac{\log \log r}{\log r}}$ in \eqref{eq:zurbound}. We then have
\begin{equation}\label{eq:valsinexp}
-n \log R = n \frac{\log \log r}{\log r} \log t, \qquad R^L \le  t^{-\frac{\log \log r}{\log r} (2 \log_t r - 1)} =  \frac{t^{ \frac{\log \log r}{\log r}}}{(\log r)^2} \le \frac{t}{(\log r)^2}.
\end{equation}
Assuming further $r \ge t^{\log^2 t}$, we have $\frac{\log t \log \log r}{\log r} \in (0,1)$, which implies $R \le 1-\log t \log \log r/(2\log r)$, and so
\begin{equation}\label{eq:valsinexp2}
\frac{1}{(L+1)(1-R)} \le \frac{\log t}{2 \log r} \frac{2 \log r}{\log t \log \log r } \le 1.
\end{equation}
Plugging \eqref{eq:valsinexp} and \eqref{eq:valsinexp2} in \eqref{eq:zurbound}, we obtain \eqref{eq:coolbound}. To prove \eqref{eq:coolbound2}, use \eqref{eq:coolbound} if $r \ge \max\{ 200000,t^{\log^2 t}, \sqrt{n}\}$, and otherwise use \eqref{eq:binombound} together with the bound $\binom{n+k}{n}\le (n+k)^{\min\{n,k\}}$.
\end{proof}

\section{Proof of Theorem \ref{thm:expothm}}\label{sec:proofl1}
By \eqref{eq:triang}, it suffices to bound
\begin{equation}
 \sum_{\substack{\FactType \in \Omega \\ |\FactType|=n}} |S(n,\chi\cdot \Ind_{\FactType})|,
\end{equation}
where $\Ind_{\FactType}$ is the indicator function of polynomials $f$ with $\FactType_f=\FactType$ (see \eqref{eq:snalpha} for the definition of $S$). Let
\begin{equation}
\Omega_d \subseteq \Omega
\end{equation}
be the subset of factorization types containing only pairs $(x,y)\in \mathbb{N}^2$ with $x=d$. The elements of $\Omega_d$, for each $d$, may be parametrized by partitions -- to each $\lambda=(\lambda_1,\lambda_2, \ldots, \lambda_k) \in \YY$, we associate 
\begin{equation}
\omega_{\lambda,d}:=\{ (d,\lambda_1),\ldots,(d,\lambda_k)\} \in \Omega_d.
\end{equation}
Note that $|\omega_{\lambda,d}| = d|\lambda|$. By Lemma \ref{lem:multi}, $\sum_{\FactType \in \Omega, \, |\FactType|=n} |S(n,\chi\cdot \Ind_{\FactType})|$ is the coefficient of $u^n$ in the following power series:
\begin{equation}
F_{\chi}(u):= \prod_{d \ge 1} F_d(u^d)
\end{equation}
where
\begin{equation}
F_d(u):=\sum_{\lambda \in \YY}   \left|S(d|\lambda|,\chi\cdot \Ind_{\omega_{\lambda,d}})\right| u^{|\lambda|} . 
\end{equation}
The terms with $\ell(\lambda)>|\mathcal{P}_d|$ do not contribute to $F_d(u)$, as there is no factorization type with more than $|\mathcal{P}_d|$ distinct primes of degree $d$. By Lemma \ref{lem:monomial} and \eqref{eq:basechange}, for each $\lambda \in \YY$ and $d \ge 1$, we have
\begin{equation}\label{eq:stoc}
S(d|\lambda|,\chi\cdot \Ind_{\omega_{\lambda,d}}) = \sum_{\mu \vdash |\lambda|} c_{\lambda,\mu} p_{\mu}(\chi(P): P \in \mathcal{P}_d).
\end{equation}
Let $\ord \ge 2$ be the order of $\chi$. Writing $\mu$ as $(\mu_1,\mu_2,\ldots,\mu_r)$, we may bound $p_{\mu}(\chi(P): P \in \mathcal{P}_d)$ using Lemma \ref{lem:PrimeBound} as follows:
\begin{equation}\label{eq:pbound}
\begin{split}
\left| p_{\mu}(\chi(P): P \in \mathcal{P}_d) \right| &= \prod_{i=1}^{r} \left| p_{\mu_i}(\chi(P): P \in \mathcal{P}_d) \right| \\
&\le \prod_{\substack{1 \le i \le r \\ \ord \nmid \mu_i}} \left( \min\{ \frac{q^{\frac{d}{2}}}{d}(\ell+\deg(M)+1),\frac{q^d}{d} \} \right) \prod_{\substack{1 \le i \le r \\ \ord \mid \mu_i}} \left( \frac{q^{d}}{d}\right).
\end{split}
\end{equation}
From \eqref{eq:stoc}, \eqref{eq:pbound}, \eqref{eq:deff} and \eqref{eq:cdef}, we have for all $n \ge 1$
\begin{equation}\label{eq:stoexpf}
\begin{split}
\sum_{\substack{\lambda \vdash n\\ \ell(\lambda) \le |\mathcal{P}_d|}} \left| S(d|\lambda|,\chi\cdot \Ind_{\omega_{\lambda,d}}) \right| &\le \sum_{\substack{\lambda, \mu \vdash n \\ \\\ell(\lambda) \le |\mathcal{P}_d|}} |c_{\lambda,\mu}| \prod_{\substack{1 \le i \le \ell(\mu) \\ \ord \nmid \mu_i}} 
\left(\min\{ \frac{q^{\frac{d}{2}}}{d}(\ell+\deg(M)+1),\frac{q^d}{d} \}\right)
\prod_{\substack{1 \le i \le \ell(\mu) \\ \ord \mid \mu_i}} \left( \frac{q^{d}}{d}\right)\\
&= \mathrm{E}_{\pi \in S_n} f(\pi)  \Big(\sum_{\substack{\lambda \vdash n\\\ell(\lambda) \le |\mathcal{P}_d |}}w(B_{\lambda,\mu_{\pi}}) \Big),
\end{split}
\end{equation}
where $f$ is defined as in \eqref{eq:deff} with $m:=\ord$ and
\begin{equation}
z_1 = z_{1,d}:= \min\{\frac{q^{\frac{d}{2}}}{d} (\ell+\deg(M)+1),\frac{q^{d}}{d}\}, \qquad z_2= z_{2,d}:=\frac{q^d}{d},
\end{equation}
and $\mu_{\pi}$ is the partition of $n$ whose parts are the cycle sizes of $\pi$. By Lemma~\ref{lem:bsum} and \eqref{eq:stoexpf}, the coefficients of $F_d$ are bounded from above by the coefficients of
\begin{equation}
G_d(u):=1+\sum_{n \ge 1} \Big(\sum_{i=0}^{|\mathcal{P}_d|} \binom{n}{i} \Big) \mathrm{E}_{\pi \in S_n} f(\pi)  u^n.
\end{equation}
If $q \ge 4$, we do the following. Replacing $\sum_{i=0}^{|\mathcal{P}_q|} \binom{n}{i}$ with $2^n$, we use Lemma \ref{lem:meanz} to bound the coefficients of $G_d$ from above by the coefficients of
\begin{equation}
H_d(u):= (1-2u)^{-z_{1,d}}(1-(2u)^{\ord})^{(-z_{2,d}+z_{1,d})/\ord},
\end{equation}
and so the coefficients of $F_{\chi}$ are bounded from above by the coefficients of
\begin{equation}
H_{\chi}(u):= \prod_{d \ge 1} H_d(u^d).
\end{equation}
Summarizing,
\begin{equation}\label{eq:stohprodd}
\begin{split}
\sum_{\FactType \in \Omega, \, |\FactType|=n} |S(n,\chi\cdot \Ind_{\FactType})| &= [u^n] F_{\chi}(u) \le [u^n] H_{\chi}(u) \\
&= [u^n] \prod_{d \ge 1} (1-2u^d)^{-z_{1,d}} (1-2^{\ord}u^{d\, \ord})^{(-z_{2,d}+z_{1,d})/\ord}.
\end{split}
\end{equation}
The logarithm of the power series $H_{\chi}$ is given by 
\begin{equation}
\log H_{\chi}(u) = \sum_{i,d \ge 1} \frac{2^iu^{di}}{i} z_{1,d} + \sum_{i,d \ge 1} \frac{2^{i\, \ord} u^{di\, \ord}}{i} \frac{z_{2,d}-z_{1,d}}{\ord},
\end{equation}
so that 
\begin{equation}
[u^k]\log H_{\chi}(u) = \frac{1}{k}\sum_{di = k} 2^i (z_{1,d}d) + \frac{\Ind_{\ord \mid k}}{k} \cdot  \sum_{di = \frac{k}{\ord}}2^{i\, \ord}(z_{2,d}-z_{1,d})d.
\end{equation}
Set
\begin{equation}
L:=\lfloor 2\log_q (\ell+\deg(M)+1) \rfloor.
\end{equation}
For $d \le L$, we have $z_{1,d}=z_{2,d}=q^d/d$, so that by the first part of Lemma \ref{lem:bounds} with $x=q$,
\begin{equation}
[u^k]\log H_{\chi}(u) \le \frac{1}{k}\sum_{di = k} 2^i q^d \le \frac{8q^k}{k}.
\end{equation}
for all $1\le k \le L$. As $z_{2,d}-z_{1,d} \le q^d/d$ and $z_{1,d} \le q^{\frac{d}{2}} (\ell +\deg(M)+1)/d$, we also have, by the first part of Lemma~\ref{lem:bounds} with $x=\sqrt{q}$ and the second part of the lemma with $x=q$, that
\begin{equation}\label{eq:logh}
\begin{split}
[u^k] \log H_{\chi}(u) &\le \frac{(\ell + \deg(M)+1)}{k}\sum_{di=k} 2^i q^{\frac{d}{2}} + \frac{\Ind_{\ord \mid k}}{k} \cdot \sum_{di = \frac{k}{\ord}} 2^{i\, \ord} q^d \\
&\le 8(\ell + \deg(M)+1) \frac{q^{\frac{k}{2}}}{k} + 10\frac{q^{\frac{k}{2}}}{k}\le 10(\ell + \deg(M)+2)\frac{q^{\frac{k}{2}}}{k}
\end{split}
\end{equation}
for all $k \ge 1$. From \eqref{eq:stohprodd} and \eqref{eq:logh}, we have
\begin{equation}
\sum_{\FactType \in \Omega, \, |\FactType|=n} |S(n,\chi\cdot \Ind_{\FactType})| \le [u^n] \exp( \sum_{1 \le k \le L} \frac{10q^k}{k} u^k+ \sum_{ k > L}10(\ell + \deg(M)+2) \frac{q^{\frac{k}{2}}}{k}u^k),
\end{equation}
and by Lemma \ref{lem:genfunc} with $t=q$, $r=\ell+\deg(M)+1$, we have the estimates 
\begin{equation}
\sum_{\FactType \in \Omega, \, |\FactType|=n} |S(n,\chi\cdot \Ind_{\FactType})| \le q^{\frac{n}{2}} \binom{20(\ell+\deg(M)+2)+n-1}{n}
\end{equation}
and, if $\ell +\deg(M) = O(n)$,
\begin{equation}
\sum_{\FactType \in \Omega, \, |\FactType|=n} |S(n,\chi\cdot \Ind_{\FactType})| \le q^{\frac{n}{2}} e^{O_q(\frac{n \log \log (n+2)}{\log(n+2)})}.
\end{equation}
The theorem is established for $q \ge 4$. We now suppose $q \in \{2,3\}$. We define $\widetilde{H_d}:=H_d$ for $d \ge 2$, while for $d=1$
\begin{equation}
\widetilde{H_1}(u):= 1+ \sum_{n \ge 1} 1.4^n \mathrm{E}_{\pi \in S_n} f(\pi) u^n = (1-1.4u)^{-z_{1,1}}(1-(1.4u)^{\ord})^{(-z_{2,1}+z_{1,1})/\ord},
\end{equation}
where in the last passage we used Lemma \ref{lem:meanz}. As $|\mathcal{P}_1 | =q\le 3$ for $q \in \{2,3\}$, it follows that $\sum_{i=0}^{|\mathcal{P}_1|} \binom{n}{i}\le \sum_{i=0}^{3} \binom{n}{i}$, which is at most $7\cdot 1.4^n$ by Lemma \ref{lem:binomtoexp}. Thus the coefficients of $G_1$ are bounded from above by those of $\widetilde{H_1}$ times $7$, and so the coefficients of $F_{\chi}$ are bounded from above by those of
\begin{equation}
\widetilde{H_{\chi}}(u) := \prod_{d \ge 1} \widetilde{H_{d}}(u^d),
\end{equation}
times $7$. As in the case $q \ge 4$, we proceed to upper bound the coefficients of $\log \widetilde{H_{\chi}}$. For $d \le L$, we have $z_{1,d}=z_{2,d}=q^d/d$, so that by Lemma \ref{lem:bounds2} with $x=q$,
\begin{equation}\label{eq:logtildehsmall}
 [u^k] \log \widetilde{H_{\chi}}(u) \le  \frac{1}{k}\Big(1.4^k q + \sum_{di=k,\, d \neq 1} 2^i q^{d}\Big)\le 10\frac{q^{k}}{k},
\end{equation}
for all $1 \le k \le L$, where we used $1.4<q\le 3$. For any $k \ge 1$, we have by Lemma \ref{lem:bounds2} that
\begin{equation}\label{eq:logtildeh}
\begin{split}
[u^k] \log \widetilde{H_{\chi}}(u) &\le \frac{(\ell + \deg(M)+1)}{k}\Big( 1.4^k q^{\frac{1}{2}} + \sum_{di=k,\, d \neq 1} 2^i q^{\frac{d}{2}} \Big)+ \frac{\Ind_{\ord \mid k}}{k} \big( \sum_{di = \frac{k}{\ord},\, d \neq 1} 2^{i\, \ord} q^d + 1.4^k q\big) \\
&\le \frac{(\ell + \deg(M)+1)}{k}\Big(2 \cdot 1.4^k + 7 q^{\frac{k}{2}} \Big) + 14\frac{q^{\frac{k}{2}}}{k}\le 20(\ell + \deg(M)+2)\frac{q^{\frac{k}{2}}}{k}.
\end{split}
\end{equation}
From this point we continue as in the case $q \ge 4$ and conclude the proof of the theorem.

\section*{Acknowledgements}
The author would like to express his gratitude to Darij Grinberg for pointing out to him the work of E\u{g}ecio\u{g}lu and Remmel \cite{ege1991}. The author would like to thank Dimitris Koukoulopoulos for a discussion on $H(x,y,2y)$ and to Brad Rodgers, Ze'ev Rudnick, Will Sawin and the anonymous referees for useful comments and suggestions. The author was supported by the European Research Council under the European Union's Seventh Framework Programme (FP7/2007-2013) / ERC grant agreement no. 320755.
	
\bibliographystyle{abbrv}
\bibliography{references}

\begin{thebibliography}{10}

\bibitem{efrat2015}
E.~Bank, L.~Bary-Soroker, and L.~Rosenzweig.
\newblock Prime polynomials in short intervals and in arithmetic progressions.
\newblock {\em Duke Math. J.}, 164(2):277--295, 2015.

\bibitem{bhowmick2017}
A.~Bhowmick, T.~H. L\^{e}, and Y.-R. Liu.
\newblock A note on character sums in finite fields.
\newblock {\em Finite Fields Appl.}, 46:247--254, 2017.

\bibitem{cui2014}
Z.~Cui and J.~Wu.
\newblock The {S}elberg-{D}elange method in short intervals with an
  application.
\newblock {\em Acta Arith.}, 163(3):247--260, 2014.

\bibitem{effinger1991}
G.~W. Effinger and D.~R. Hayes.
\newblock {\em Additive number theory of polynomials over a finite field}.
\newblock Oxford Mathematical Monographs. The Clarendon Press, Oxford
  University Press, New York, 1991.
\newblock Oxford Science Publications.

\bibitem{ege1991}
O.~E\u{g}ecio\u{g}lu and J.~B. Remmel.
\newblock Brick tabloids and the connection matrices between bases of symmetric
  functions.
\newblock {\em Discrete Appl. Math.}, 34(1-3):107--120, 1991.
\newblock Combinatorics and theoretical computer science (Washington, DC,
  1989).

\bibitem{Feng2017}
B.~Feng and Z.~Cui.
\newblock {DDT} theorem over square-free numbers in short interval.
\newblock {\em Frontiers of Mathematics in China}, 12(2):367--375, Apr 2017.

\bibitem{ford20082}
K.~Ford.
\newblock The distribution of integers with a divisor in a given interval.
\newblock {\em Ann. of Math. (2)}, 168(2):367--433, 2008.

\bibitem{ford2008}
K.~Ford.
\newblock Integers with a divisor in {$(y,2y]$}.
\newblock In {\em Anatomy of integers}, volume~46 of {\em CRM Proc. Lecture
  Notes}, pages 65--80. Amer. Math. Soc., Providence, RI, 2008.

\bibitem{hall1982}
R.~R. Hall and G.~Tenenbaum.
\newblock On the average and normal orders of {H}ooley's {$\Delta $}-function.
\newblock {\em J. London Math. Soc. (2)}, 25(3):392--406, 1982.

\bibitem{hayes1965}
D.~R. Hayes.
\newblock The distribution of irreducibles in {${\rm GF}[q,\,x]$}.
\newblock {\em Trans. Amer. Math. Soc.}, 117:101--127, 1965.

\bibitem{hooley1979}
C.~Hooley.
\newblock On a new technique and its applications to the theory of numbers.
\newblock {\em Proc. London Math. Soc. (3)}, 38(1):115--151, 1979.

\bibitem{katai1985}
I.~K\'{a}tai.
\newblock A remark on a paper of {K}. {R}amachandra.
\newblock In {\em Number theory ({O}otacamund, 1984)}, volume 1122 of {\em
  Lecture Notes in Math.}, pages 147--152. Springer, Berlin, 1985.

\bibitem{katai2003}
I.~K\'{a}tai and M.~V. Subbarao.
\newblock Some remarks on a paper of {R}amachandra.
\newblock {\em Liet. Mat. Rink.}, 43(4):497--506, 2003.

\bibitem{keating2016}
J.~Keating and Z.~Rudnick.
\newblock Squarefree polynomials and {M}\"obius values in short intervals and
  arithmetic progressions.
\newblock {\em Algebra Number Theory}, 10(2):375--420, 2016.

\bibitem{matomaki2016}
K.~Matom\"{a}ki and M.~Radziwi\l\l.
\newblock Multiplicative functions in short intervals.
\newblock {\em Ann. of Math. (2)}, 183(3):1015--1056, 2016.

\bibitem{meisner2018erd}
P.~Meisner.
\newblock Erd\"os' multiplication table problem for function fields and
  symmetric groups.
\newblock {\em arXiv preprint arXiv:1804.08483}, 2018.

\bibitem{ramachandra1976}
K.~Ramachandra.
\newblock Some problems of analytic number theory.
\newblock {\em Acta Arith.}, 31(4):313--324, 1976.

\bibitem{rhin1972}
G.~Rhin.
\newblock R\'epartition modulo {$1$} dans un corps de s\'eries formelles sur un
  corps fini.
\newblock {\em Dissertationes Math. (Rozprawy Mat.)}, 95, 1972.

\bibitem{rodgers2018}
B.~Rodgers.
\newblock Arithmetic functions in short intervals and the symmetric group.
\newblock {\em Algebra Number Theory}, 12(5):1243--1279, 2018.

\bibitem{rosen2002}
M.~Rosen.
\newblock {\em Number theory in function fields}, volume 210 of {\em Graduate
  Texts in Mathematics}.
\newblock Springer-Verlag, New York, 2002.

\bibitem{sawin2018}
W.~Sawin.
\newblock Square-root cancellation for sums of factorization functions over
  short intervals in function fields.
\newblock {\em arXiv preprint arXiv:1809.05137}, 2018.

\bibitem{sawin2019}
W.~Sawin and M.~Shusterman.
\newblock On the {C}howla and twin primes conjectures over {$\mathbb{F}_q[t]$}.
\newblock {\em arXiv preprint arXiv:1808.04001}, 2019.

\bibitem{stanley1999enumerative}
R.~P. Stanley.
\newblock {\em Enumerative combinatorics. {V}ol. 2}, volume~62 of {\em
  Cambridge Studies in Advanced Mathematics}.
\newblock Cambridge University Press, Cambridge, 1999.
\newblock With a foreword by Gian-Carlo Rota and appendix 1 by Sergey Fomin.

\bibitem{weil1974}
A.~Weil.
\newblock {\em Basic number theory}.
\newblock Springer-Verlag, New York-Berlin, third edition, 1974.
\newblock Die Grundlehren der Mathematischen Wissenschaften, Band 144.

\end{thebibliography}

\Addresses

\end{document}